\documentclass{article}
\usepackage[cp1251]{inputenc}
\usepackage[OT1]{fontenc}
\usepackage[english]{babel}
\usepackage[tbtags]{amsmath}
\usepackage{amsfonts, amsmath, amsthm, amssymb, graphicx, comment}
\usepackage{cite}

\def\A{\mathcal{A}}

\def\N{\mathbb N}
\def\R{{\mathbb R}}

\def\Z{\mathbb Z}

\newcommand{\Id}{\operatorname{Id}\nolimits}
\newcommand{\intt}{\operatorname{int}\nolimits}

\newcommand{\arch}{\operatorname{arch}\nolimits}

\def\eps{\varepsilon}
\def\lam{\lambda}
\def\f{\varphi}
\def\om{\omega}

\def\const{\operatorname{const}}
\def\Exp{\operatorname{Exp}}
\def\Lip{\operatorname{Lip}}
\def\arcosh{\operatorname{arcosh}}
\renewcommand\arch{\operatorname{arcosh}}

\newcommand{\pder}[2]{\frac{\partial \, #1}{\partial \, #2} }
\newcommand{\restr}[2]{\left. #1 \right|_{#2}}
\newcommand{\map}[3]{#1 \, : \, #2 \to #3}
\newcommand{\eq}[1]{$(\protect\ref{#1})$}
\newcommand{\be}[1]{\begin{equation}\label{#1}}
\newcommand{\ee}{\end{equation}}

\theoremstyle{plain}
\newtheorem{theorem}{Theorem}
\newtheorem{proposition}{Proposition}
\newtheorem{corollary}{Corollary}
\newtheorem{lemma}{Lemma}

\theoremstyle{remark}
\newtheorem{remark}{Remark}

\newcommand{\onefiglabelsizen}[4]
{
\begin{figure}[htbp]
\begin{center}
\includegraphics[height=#4cm]{#1}
\\
\parbox[t]{0.7\textwidth}{\caption{#2}\label{#3}}
\end{center}
\end{figure}
}

\newcommand{\twofiglabelsize}[8]
{
\begin{figure}[htbp]
\includegraphics[height=#4\textwidth]{#1}
\hfill
\includegraphics[height=#8\textwidth]{#5}
\\
\hfill
\parbox[t]{0.45\textwidth}{\caption{#2}\label{#3}}
\hfill
\parbox[t]{0.45\textwidth}{\caption{#6}\label{#7}}
\hfill
\end{figure}
}

\graphicspath{{figures/}}

\begin{document}

\title{
Families of Lorentzian problems on the Heisenberg group
}

\author{Yu. Sachkov}
\date{}

\maketitle

\begin{abstract}

\end{abstract}

\tableofcontents

\section{Introduction}

\section{The first family of Lorentzian problems} \label{sec:first}
\subsection{Problem statement}

Let $M = \R^3_{x, y, z}$ be the Heisenberg group with the product rule   
\begin{align*}
\begin{pmatrix}
x_1\\
y_1\\
z_1
\end{pmatrix} \cdot
\begin{pmatrix}
x_2\\
y_2\\
z_2
\end{pmatrix} = 
\begin{pmatrix}
x_1 + x_2\\
y_1 + y_2\\
z_1 + z_2 + (x_1 y_2 - x_2 y_1)/2
\end{pmatrix}.
\end{align*}
The vector fields
$$
X_1 = \pder{}{x} - \frac{y}{2}\pder{}{z}, \qquad X_2 = \pder{}{y} + \frac{x}{2}\pder{}{z}, \qquad X_3 = \pder{}{z} 
$$
form a left-invariant frame on 
 $M$. The one-forms
$$
\om_1 = dx, \qquad \om_2 = dy, \qquad \om_3 = \frac y2 dx - \frac x2 dy + dz
$$
form the dual left-invariant coframe on $M$.

Fix a parameter $\eps> 0$ and consider the Lorentzian form $\displaystyle g = - \om_1^2 + \om_2^2 + \frac{\om_3^2}{\eps^2}$. An orthonormal frame for this form is given by the vector fields $X_1$, $X_2$, $\eps X_3$. Lorentzian length maximizers for the form $g$ are solutions to the following optimal control problem:
\begin{align}
&\dot q = u_1X_1 +u_2X_2+\eps u_3X_3, \qquad q \in M, \quad u \in U = \left\{u_1 \geq \sqrt{u_2^2+u_3^2}\right\}, \label{pr1}\\
&q(0) = q_0, \qquad q(t_1) = q_1, \label{pr2}\\
&l = \int_0^{t_1} \sqrt{u_1^2 - u_2^2 - u_3^2} dt \to \max.\label{pr3}
\end{align} 
The problem is left-invariant, thus we will assume that $q_0 = \Id = (0, 0, 0)$.

The aim of this section is to study the Lorentzian problem \eq{pr1}--\eq{pr3} for any $\eps > 0$ and, as a consequence, to study its limit behaviour for $\eps \to 0$. It is natural to expect that essential objects of the Lorentzian problem (attainable set, extremal trajectories, spheres) tend to  the corresponding objects of the sub-Lorentzian problem with the orthonormal frame $\{X_1, X_2\}$ (this sub-Lorentzian problem was studied in papers \cite{groch4, groch6, sl_heis}.

\subsection{Invariant set}
Denote by $\A$ (resp. by $\A(t_1)$) the attainable set of system \eq{pr1} from the point $q_0$ for arbitrary nonnegative time (resp. for time $t_1 > 0$). 

\begin{lemma}\label{lem:x|y|}
There holds the inclusion
$$
\A \subset \{q = (x, y, z) \in M \mid x \geq |y|\},
$$
thus system \eq{pr1}, \eq{pr2} is not globally controllable.
\end{lemma}
\begin{proof}
It follows from system \eq{pr1} that
$$
\dot x = y_1, \qquad \dot y = u_2, \qquad u_1 \geq u_2,
$$
whence $|\frac{dy}{dx}| = \frac{|u_2|}{u_1} \leq 1$, so $x(t) \geq |y(t)|$.
\end{proof}

In order to construct an invariant set (which will turn out to be the attainable set $\A$) we apply the Pontryagin maximum principle in the geometric form. 

Let $T^*M$ be the cotangent bundle of the group $M$. Denote the linear on leaves of this bundle Hamiltonians $h_i(\lam) = \langle \lam, X_i\rangle$, $i = 1, 2, 3.$ Let $h_u = u_1h_1+u_2h_2+\eps u_3h_3$ be the Hamiltonian of the Pontryagin maximum principle in the geometric form.

\begin{theorem}[\cite{notes}]
Let $q(t)$, $t \in [0, t_1]$, be a trajectory of system \eq{pr1}, \eq{pr2} that satisfies the inclusion $q(t_1) \in \partial \A(t_1)$ and corresponds to a control $u(t)$, then there exists a curve $\lambda \in \Lip([0, t_1], T^*M)$, $\lambda_t \in T_{q(t)}^*M$, for which the following conditions hold for a.e. $t\in[0, t_1]$:
\begin{align}
&\dot \lambda_t = \vec{h}_{u(t)}(\lambda_t), \label{Ham}\\
&h_{u(t)}(\lambda_t) = \max_{v \in U} h_v(\lambda_t), \label{max}\\
&\lam_t \neq 0, \label{nontriv}
\end{align}
where $\vec h_u$ is the Hamiltonian vector field on $T^*M$ corresponding to the Hamiltonian $h_u$.
\end{theorem}

The maximality condition \eq{max} implies that
\be{h_1}
h_1 = - \sqrt{h_2^2 + \eps^2 h_3^2}, \qquad (u_1, u_2, u_3) = (-h_1, h_2, \eps h_3),
\ee
and the vertical subsystem of the Hamiltonian system \eq{Ham} (the subsystem for adjoint variables) reads
$$
\dot h_1 = - u_2 h_3, \qquad \dot h_2 = u_1 h_3, \qquad \dot h_3 = 0.
$$
Then, in view of the second equality in \eq{h_1}, we get the Hamiltonian system
\begin{align*}
&\dot h_1 = - h_2h_3, \qquad \dot h_2 = - h_1h_3, \qquad \dot h_3 = 0, \\
&\dot x = - h_1, \qquad \dot y = h_2, \qquad \dot z = h_1\frac y2 + h_2 \frac x2 + h_3^2\eps. 
\end{align*}
If $h_3 = 0$, then $h_1, h_2 \equiv \const$ and
$$
x = - h_1 t, \qquad y = h_2 t, \qquad z = 0.
$$
Let $h_3 \neq 0$, then
\begin{align}
&h_1 = h_1^0 \cosh \tau - h_2^0\sinh \tau, \qquad h_2 = h_2^0 \cosh\tau-h_1^0\sinh\tau, \qquad \tau = h_3 t, \nonumber\\
&x = (-h_1^0\sinh\tau +h_2^0(\cosh\tau-1))/h_3, \label{x}\\
&y = (h_2^0\sinh\tau - h_1^0(\cosh\tau-1))/h_3, \label{y}\\
&z = \eps^2(\sinh \tau + \tau)/2.
\end{align}
Immediate computation on the basis of \eq{x}, \eq{y} gives
$$
\frac{x^2-y^2}{2\eps^2}+1=\cosh\tau,
$$
thus 
$$
\tau=\pm\arcosh\left(\frac{x^2-y^2}{2\eps^2}+1\right).
$$
With account of Lemma \ref{lem:x|y|} we proved the following
\begin{proposition}
If a trajectory $q(t)$, $t \in [0, t_1]$, comes to $\partial \A(t_1)$ at a point $q = (x, y, z)$, then 
\begin{multline}\label{|z|}
|z| = \frac{\eps^2}{2}(\sinh \tau+\tau), \qquad x \geq |y|, \qquad 
\tau=\arcosh\left(\frac{x^2-y^2}{2\eps^2}+1\right).
\end{multline}
\end{proposition}
Denote the surface obtained:
\begin{multline*}
S = \left\{q = (x, y, z) \in M \mid  x \geq |y|, \ |z| = \frac{\eps^2}{2}(\sinh \tau+\tau), \right.\\
\left. \tau=\arcosh\left(\frac{x^2-y^2}{2\eps^2}+1\right) \right\},
\end{multline*}
and a domain in $M$ bounded by it:
\begin{multline*}
D = \left\{q = (x, y, z) \in M \mid  x \geq |y|, \ |z| \leq \frac{\eps^2}{2}(\sinh \tau+\tau), \right.\\
\left. \tau=\arcosh\left(\frac{x^2-y^2}{2\eps^2}+1\right) \right\}.
\end{multline*}

\begin{proposition}\label{prop:invar}
\begin{itemize}
\item[$(1)$]
The domain $D$ is an invariant set of system \eq{pr1}.
\item[$(2)$]
$D \supset \A$.
\end{itemize}
\end{proposition}
\begin{proof}
(1) Let us prove that if a trajectory $q(t)$ of system \eq{pr1} satisfies the inclusion $q(0) \in D$, then $q(t) \in D$ for all $t > 0$.
To this end we show that the scalar product $\restr{n \cdot \dot q}{S} \leq 0$, where $n$ is the exterior normal vector to the set $D$. 

By virtue of the symmetry $z \mapsto -z$, we may consider only the points of the surface $S_+ = \{q \in S \mid z \geq 0\}$. The exterior normal vector to this surface is
$$
n = \left(-\frac x2 \frac{\cosh\tau+1}{\sinh\tau}, \frac y2 \frac{\cosh\tau+1}{\sinh\tau}, 1\right).
$$
Thus
\begin{align*}
2 n \cdot \dot q &= \left(-\frac x2 \frac{\cosh\tau+1}{\sinh\tau}, \frac y2 \frac{\cosh\tau+1}{\sinh\tau}, 1\right)\cdot
\left(u_1, u_2, - \frac y2 u_1 + \frac x2 u_2 + \eps u_3\right) \\
&= - u_1\left(x \frac{\cosh\tau+1}{\sinh\tau}+y\right) + u_2 \left(y \frac{\cosh\tau+1}{\sinh\tau} + x\right) + 2 \eps u_3 \\
&= (-u_1, u_2, u_3) \cdot(v_1, v_2, v_3), \\
(v_1, v_2, v_3) &= 
\left( 
x \frac{\cosh\tau+1}{\sinh\tau}+y, 
y \frac{\cosh\tau+1}{\sinh\tau} + x, 
2 \eps\right).
\end{align*}
In view of the conditions $-u_1 \leq - \sqrt{u_2^2+u_3^2}$, $v _1 = \sqrt{v_1^2+v_2^2}$, we have
$$
2 n\cdot\dot q = (u_1, u_2, u_3)\cdot (v_1, v_2, v_3) \leq 0.
$$
So at the surface $S_+$ the vector $\dot q$ is directed inside of the domain $D$ or is tangent to its boundary $S$. Thus this domain is an invariant set of system \eq{pr1}.

(2) We have $q_0 = (0, 0, 0) \in D$, thus with account of item (1) we get $\A \subset D$.
\end{proof}

The surface $S$ for $\eps=1$ is shown in Fig. \ref{fig:S}. In Sec. \ref{subsec:attain} we prove the equality $\A = D$.

\onefiglabelsizen{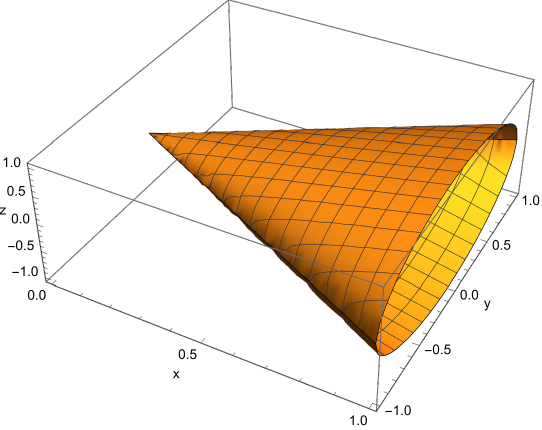}{The surface $S$ for $\eps=1$}{fig:S}{5}

\subsection{Extremal trajectories}
Now we apply the Pontryagin maximum principle \cite{PBGM, notes} to problem \eq{pr1}--\eq{pr3}.

Abnormal trajectories are lightlike and fill the surface $S$.

Normal arclength parametrized extremals are trajectories of the Hamiltonian vector field $\vec H$ with the Hamiltonian $H = \frac 12(-h_1^2+h_2^2+\eps^2 h_3^2)$ lying in the level surface $\{H = - \frac 12\}$. The Hamiltonian system $\dot \lambda = \vec H(\lambda)$, $\lambda \in \{H = - \frac 12\}$, has the form
\begin{align*}
&\dot h_1 = - h_2 h_3, \qquad \dot h_2 = - h_1 h_3, \qquad \dot h_3 = 0, \\
&\dot x = -h_1, \qquad \dot y = h_2, \qquad \dot z = h_1 \frac y2 + h_2 \frac x2 + \eps^2 h_3.
\end{align*}
If $h_3 = 0$, then 
\be{xyz0}
x = - h_1t, \qquad y = h_2t, \qquad z = 0.
\ee
If $h_3 \neq 0$, then
\begin{align}
&h_1 = h_1^0 \cosh \tau-h_2^0\sinh\tau, \qquad h_2=h_2^0\cosh\tau-h_1^0\sinh\tau, \qquad \tau = h_3 t, \nonumber\\
&x = (-h_1^0 \sinh \tau+h_2^0(\cosh\tau-1))/h_3, \nonumber\\
&y = (h_2^0\sinh\tau- h_1^0(\cosh\tau-1))/h_3, \nonumber\\
&z = \frac{1}{2h_3^2}(\sinh\tau-\tau)+\frac{\eps^2}{2}(\sinh\tau+\tau). \label{zt}
\end{align}

Let us parametrize the level surface $\{H = - \frac 12\}$:
\begin{align*}
&h_1 = - \cosh \theta, \qquad h_2 = \sinh \theta \cos \f, \qquad \eps h_3 = \sinh \theta \sin \f, \\
&\theta \in \R, \quad \f \in \R/(2 \pi\Z).
\end{align*}
Then all normal extremal trajectories are parametrized by the exponential mapping
\begin{align*}
&\map{\Exp}{\left(T_{q_0}^*M \cap \left\{H = - \frac 12\right\}\right)\times \R_{+t}}{M}, \\
&\Exp(\theta, \f, t) = (x, y, z).
\end{align*} 

\subsection{Conjugate points}

\begin{proposition}\label{prop:conj}
Problem \eq{pr1}--\eq{pr3} has no conjugate points in the domain $\{h_3 \neq 0\}$.
\end{proposition}
\begin{proof}
In the domain $\left\{h_3 = \frac{\sinh \theta \sin \f}{\eps} \neq 0\right\}$ we have
$$
J = 
\frac{\partial(x, y, z)}{\partial(t, \theta, \f)} = 
\frac{\eps^4}{\sin^3 \f \sinh^2 \theta} 
\left(\tau \sin \tau + \frac{2 - \cosh \tau + \tau \sinh \tau}{\sin^2 \f\sinh^2\theta}\right).
$$
When $\tau \neq 0$ the function $2 - \cosh \tau + \tau \sinh \tau$ is positive, thus $J \neq 0 $ for $h_3 \neq 0$. Consequently, there are no conjugate points for $h_3 \neq 0$.
\end{proof} 

\begin{corollary}
For $h_3 \neq 0$ extremal trajectories $\Exp(\theta, \f, t)$ are locally optimal.
\end{corollary}

\subsection{Diffeomorphic property of the exponential mapping}

Denote a domain in the preimage of the exponential mapping
$$
N = \{(\theta, \f, \tau) \mid \theta > 0, \ \f \in (0, \pi), \ \tau > 0\},
$$
and a subset of the invariant set
\begin{multline*}
D_+ = \left\{q = (x, y, z) \in M \mid x > y, \  0 < z < \frac{\eps^2}{2} (\sinh\tau+\tau), \right. \\
\left. \tau = \arcosh\left( \frac{x^2-y^2}{2\eps^2}+1\right)\right\}.
\end{multline*}

\begin{proposition}\label{prop:diff}
\begin{itemize}
\item[$(1)$]
$\Exp(N) \subset D_+$.
\item[$(2)$]
The mapping $\map{\Exp}{N}{D+}$ is a diffeomorphism. 
\end{itemize}
\end{proposition}
\begin{proof}
(1) The set $D$ is invariant, thus $\Exp(N) \subset D$.

It follows from equality \eq{zt} that $z > 0$ for $\tau>0$. Let us prove that $\Exp(N) \cap \partial D\neq 0$. By contradiction, suppose that there exists a point $q \in \Exp(N) \cap \partial D$. Then $q = \Exp(\theta, \f, \tau)$ for some $\nu = (\theta, \f, \tau) \in N$. By virtue of Proposition \ref{prop:conj}, the mapping $\Exp$ is a local diffeomorphism at the point $\nu$, thus $\Exp(N)$ contains a neighborhood of the point $q$. Thus $\Exp(N)\not\subset D$, a contradiction.

So $\Exp(N) \subset \intt D \cap \{z > 0\}$, thus $\Exp(N) \subset D_+$.

(2) Let us apply Hadamard's global diffeomorphism theorem \cite{had}: if $\map{F}{X}{Y}$ is a smooth nondegenerate proper mapping between connected simply connected smooth manifolds of the same dimension, then $F$ is a diffeomorphism. Let us apply this theorem to the mapping $\map{\Exp}{N}{D_+}$. 

It is obvious that $N$ and $D_+$ are 3D smooth connected simply connected manifolds.

The mapping $\map{\Exp}{N}{D_+}$ is nondegenerate by Proposition \ref{prop:conj}.

Let us prove that the mapping $\map{\Exp}{N}{D_+}$ is proper, i.e., for any compact $K\subset D_+$ the preimage $\Exp^{-1}(K) \subset N$ is compact. In order to prove this we show that if a sequence $\nu_n = (\theta_n, \f_n, \tau_n)\to \partial N$, then its image $\Exp(\nu_n) \to \partial D_+$.

If $\nu_n = (\theta_n, \f_n, \tau_n)\to \partial N$, then there exists a subsequence on which one of the following conditions hold:

1) $\theta \to 0$, 

2) $\theta \to \infty$,

3) $\f \to 0$, 

4) $\f \to \pi$,

5) $\tau \to 0$, 

6) $\tau \to + \infty$. 

Let us consider these possibilities case by case and show that in each of them $q = (x, y, z) = \Exp(\theta, \f, \tau) \to \partial D_+$.

1) Let $\theta \to 0$, then $h_3 \to 0$.

1.1) Let $\tau \to \bar{\tau} \in (0, + \infty)$. Then $z \to + \infty$.

1.2) Let $\tau \to + \infty$. Then $z \to + \infty$.

1.3) Let $\tau \to 0$. Then $\sinh \tau - \tau \sim \frac{\tau^3}{6}$.

1.3.1) Let $\frac{\tau^3}{h_3^2}\to 0$, then $z \to 0$.

1.3.2) Let $\frac{\tau^3}{h_3^2}\to + \infty$, then $z \to + \infty$.

1.3.3) Let $\frac{\tau^3}{h_3^2}\to C \in (0, + \infty)$, then $x \to \infty$.

2) Let $\theta \to + \infty$.

2.1) Let $\tau \to 0$.

2.1.1) Let $h_3 \to 0$, then $z \to 0$.

2.1.2) Let $h_3 \to \bar h_3\in(0. + \infty)$, then $z \to 0$.

2.1.3) Let $h_3 \to 0$.

2.1.3.1) Let $\frac{\tau^3}{h_3^2} \to 0$, then $z \to 0$.

2.1.3.2) Let $\frac{\tau^3}{h_3^2} \to \infty$, then $z \to +\infty$.

2.1.3.3) Let $\frac{\tau^3}{h_3^2} \to C \in (0, +\infty)$, then $x \to +\infty$.

2.2) Let $\tau \to + \infty$, then $z \to +\infty$.

2.3) Let $\tau \to \bar \tau \in (0, + \infty)$.

2.3.1) Let $h_3 \to 0$, then $x \to +\infty$.

2.3.2) Let $h_3 \to + \infty$.

2.3.2.1) Let $\frac{h_1}{h_3} \to + \infty$, then $x \to +\infty$.

2.3.2.2) Let $\frac{h_1}{h_3} \to C \in (0,+ \infty)$. Then
\begin{align*}
&x \to \frac{\eps \sinh \bar \tau}{\sin \bar \f} + \eps(\cosh \bar \tau - 1) \frac{\cot \bar \f}{\eps} =: \bar x, \\
&y \to \eps \cot \bar \f \sinh \bar \tau + \eps \frac{\cosh \bar \tau-1}{\sin \bar f} = : \bar y, \\
&z \to \frac{\eps^2}{2}(\sinh \bar \tau + \bar \tau) =: \bar z, \\
&\bar \tau \in (0, + \infty), \qquad \bar{\f} \in (0, \pi).
\end{align*}
Thus $\cosh \bar \tau = \frac{\bar x^2 - \bar y^2}{2 \eps^2}+1$, i.e., $q = (x, y, z) \to \bar q = (\bar x, \bar y, \bar z) \in \partial D_+$.

2.3.3) Let $h_3 \to \bar h_3 \in (0, + \infty)$.

2.3.3.1) Let $\f \to \bar \f \in (0, \frac{\pi}{2})$, then $x \to + \infty$.

2.3.3.2) Let $\f \to  \frac{\pi}{2}$, then $x \to + \infty$.

2.3.3.3) Let $\f \to \bar \f \in (\frac{\pi}{2}, \pi)$, then $x \to + \infty$ or $y \to + \infty$.

The remaining cases 3)--6) are considered similarly.

Thus the mapping $\map{\Exp}{N}{D_+}$ is proper. By Hadamard's theorem, this mapping is a diffeomorphism.

\end{proof}

\subsection{Attainable sets}\label{subsec:attain}

\begin{theorem}\label{th:attain}
There holds the identity $\A = D$.
\end{theorem}
\begin{proof}
In Proposition \ref{prop:invar} we proved the inclusion $\A \subset D$, let us prove the reverse inclusion.

Proposition \ref{prop:diff} gives the equality $\Exp(N) = D_+$, whence $D_+ \subset \A$. By virtue of the symmetry $z \to - z$, we get the inclusion $D_- \subset \A$, where $D_-$ is the image of the domain $D_+$ for this symmetry. It follows from the parametrisation of extremal trajectories for $h_3 = 0$ \eq{xyz0} that 
$$
\{q = (x, y, z) \in M \mid x > |y|, \ z = 0\} \subset \A.
$$
Thus $\intt D \subset \A$. The boundary of the set $D$ is filled by abnormal (lightlike) trajectories, thus $\partial D \subset \A$.

So $D \subset \A$, thus $D = \A$.
\end{proof}

Denote by $\A^-$ the attainable set of system \eq{pr1}, \eq{pr2} from the point $q_0$ for arbitrary nonpositive time (causal past of the point $q_0$). 

\begin{theorem}\label{th:attain-}
There holds the equality
\begin{multline*}
\A^- = \left\{ q = (x, y,z)\in M \mid x \leq y, \ |z| \leq \frac{\eps^2}{2}(\sinh \tau + \tau), \right. \\
\left. \tau = \arch\left(\frac{x^2-y^2}{2 \eps^2} + 1\right)\right\}.
\end{multline*}
\end{theorem}
\begin{proof}
Control system \eq{pr1} has a symmetry $(x, y, z, t, u_3) \mapsto (-x, -y, -z, -t, -u_3)$. We apply this symmetry to the attainable set $\A = D$ and get the statement of the theorem.
\end{proof}

Denote by $\A_q$ (resp. by $\A_q^-$) the attainable set of system \eq{pr1} from a point $q \in M$ for arbitrary nonnegative (resp. nonpositive) time.

\begin{theorem}\label{th:Aq1}
Let $q_1 = (x_1, y_1, z_1) \in M$. Then there hold the equalities
\begin{align*}
&\A_{q_1} = q_1 \cdot \A = 
\left\{ 
\vphantom{\left[\frac{(x-x_1)^2- (y - y_1)^2}{2 \eps^2}+1\right]}
q = (x, y, z) \in M \mid x - x_1 \geq |y - y_1|, \right. \\
&\qquad \left| z - z_1 - \frac{x_1  y- x y_1}{2}\right| \leq \frac{\eps^2}{2}(\sinh \tau + \tau), \\ 
&\qquad \left. \tau = \arch \left[\frac{(x-x_1)^2- (y - y_1)^2}{2 \eps^2}+1\right]\right\},\\
&\A_{q_1}^- = q_1 \cdot \A^- = 
\left\{ 
\vphantom{\left[\frac{(x-x_1)^2- (y - y_1)^2}{2 \eps^2}+1\right]}
q = (x, y, z) \in M \mid x - x_1 \leq -|y - y_1|, \right. \\
&\qquad \left| z - z_1 - \frac{x_1  y- x y_1}{2}\right| \leq \frac{\eps^2}{2}(\sinh \tau + \tau), \\ 
&\qquad \left. \tau = \arch \left[\frac{(x-x_1)^2- (y - y_1)^2}{2 \eps^2}+1\right]\right\}.
\end{align*}
\end{theorem}
\begin{proof}
Equalities $\A_{q_1} = q_1 \cdot \A$ and $\A_{q_1}^- = q_1\cdot \A^-$ follow since the Lorentzian structure is left-invariant. The explicit formulas for $\A_{q_1}$ and $\A_{q_1}^-$ follow from the explicit form of the sets $\A$  and $\A^-$. 
\end{proof}

\subsection{Existence of optimal trajectories}

\begin{theorem}\label{th:exist}
If $q_1 \in \A$, then there exists an optimal trajectory in problem \eq{pr1}--\eq{pr3}.
\end{theorem}
\begin{proof}
Let us apply the following sufficient condition for existence of optimal trajectories.

\begin{theorem}[Theorem 2 \cite{sl_exist}]\label{th:sl-exist}
Let problem \eq{pr1}--\eq{pr3} satisfy the following conditions:
\begin{enumerate}
\item
$q_1 \in \A$,
\item
the set $\A \cap \A_{q_1}^-$ is compact,
\item
$T(q_1) < + \infty$.
\end{enumerate}
Then there exists an optimal trajectory in problem \eq{pr1}--\eq{pr3}.
\end{theorem}
Here
\begin{align}
&T(q_1) = \sup \{t_1 > 0 \mid \exists \text{ trajectory $q(t)$ of system  \eq{pr21}--\eq{pr24}}, \ t \in [0, t_1], \\
&\qquad\qquad\qquad \text{ such that } q(0) = q_0, \ q(t_1) = q_1\}, \nonumber\\
&\dot q = X_1 + u_2 X_2 + \eps u_3 X_3, \qquad q \in M, \quad (u_2, u_3) \in \{u_2^2 + u_3^2 \leq 1 \}, \label{pr21}\\
&q(0) = q_0, \qquad q(t_1) = q_1, \qquad t_1 \text{ is free}, \nonumber \\
&l = \int_0^{t_1} \sqrt{1 - u_2^2 - u_3^2} dt \to \max. \label{pr24}
\end{align}

Let us check conditions 1--3 of Theorem \ref{th:sl-exist}.

Condition 1. is verified by the hypothesis of this theorem.

Condition 2. holds by virtue of the explicit form of the sets  $\A$ and $\A_{q_1}^-$ (see Theorems \ref{th:attain}, \ref{th:Aq1}).

Condition 3. holds since for system \eq{pr21} we have $\dot x = 1$, whence $T \leq x_1$.

Consequently, there exists an optimal trajectory in problem \eq{pr1}--\eq{pr3}.
\end{proof}

\begin{remark}
Existence of optimal trajectories holds also from the fact that the future cone $C = \{-dx^2 + dy^2 + {\eps^2} dz^2 \leq 0, \ - dx \leq 0\}$ intersects with the commutant $\R X_3(q_0)$ of the Lie algebra by the zero element only, by Corollary~1 \cite{sl_exist_pod}.
\end{remark}

\subsection{Optimality of extremal trajectories}

\begin{theorem}
Any extremal trajectory in problem \eq{pr1}--\eq{pr3} is optimal.
\end{theorem}
\begin{proof}
Let $q_1 \in \A$, then by Th. \ref{th:exist} there exists an optimal trajectory from the point $q_0$ to the point $q_1$. 

If $q_1 \in D_+$, then by Proposition \ref{prop:diff} there exists a unique extremal trajectory coming to the point $q_1$, thus it is optimal.

If $q_1 \in D_-$, then existence of optimal trajectory follows from the previous paragraph by virtue of the symmetry $z \to - z$. 

If $q_1 \in \{(x, y, z) \in M \mid x > |y|, \ z = 0\}$, then there exists a unique extremal trajectory \eq{xyz0} coming to the point $q_1$, thus it is optimal.

Finally, if $q_1 \in \partial \A$, then only lightlike trajectories with zero length functional
 come to $q_1$, thus they are optimal.
\end{proof}

\subsection{Spheres and distance}
Denote the Lorentzian distance from the point $q_0$:
$$
d(q_1) := \sup \{l(q(\cdot)) \mid q(\cdot) \text{ trajectory of  \eq{pr1}, \eq{pr2}}\},
$$
and the Lorentzian sphere of radius $r \geq 0$ with the center $q_0$:
$$
S(r) := \{q \in M \mid d(q) = r\}.
$$

\begin{theorem}\label{th:d}
\begin{itemize}
\item[$(1)$]
$\restr{d}{\partial \A} = 0$, $\restr{d}{\intt \A} \in (0, + \infty)$.
\item[$(2)$]
If $q = (x, y, 0) \in M$, then $d(q) = \sqrt{x^2-y^2}$.
\item[$(3)$]
The function $\restr{d}{\intt \A}$ is real-analytic.
\end{itemize}
\end{theorem} 
\begin{proof}
(1)
The boundary $\partial \A$ is filled only by lightlike trajectories with zero length $l$, thus $\restr{d}{\partial \A} = 0$.

If $q_1 \in \intt \A$, then there exists an optimal trajectory $\Exp(\theta, \f, t)$, $(\theta, \f) \in \R \times (\R/2 \pi \Z)$ such that $\Exp(\theta, \f, t_1) = q_1$ for some $t_1 \in (0, + \infty)$. Then $d(q_1) = t_1 \in (0, + \infty)$. 

(2) Let $q = (x, y, 0) \in M$. If $x = |y|$, then $q \in \partial \A$, and $d(q) = 0 = \sqrt{x^2-y^2}$ according to item (1). If $x > |y|$, then $q \in \intt \A$, and the optimal trajectory has the form \eq{xyz0}. Then $d(q) = \sqrt{x^2-y^2}$.

(3) For any $q \in \intt \A$ we have the following equality:
$$
d(q) = t \ : \  \exists (\theta, \f) \in \R \times \R/(2 \pi \Z) \text{ s.t. } q = \Exp(\theta, \f, t).
$$
The mapping $\map{\Exp}{N}{D}$ is real-analytic and is a proper diffeomorphism. By the implicit function theorem the distance is a real-analytic function since it satisfies the equality $q = \Exp(\theta, \f, d(q))$.
\end{proof}

\begin{theorem}\label{th:S}
\begin{itemize}
\item[$(1)$]
$S(0) = \partial \A$.
For any $r > 0$ there holds the inclusion $S(r) \subset \intt \A$.
\item[$(2)$]
For any $r \geq 0$ we have $S(r) \cap \{z = 0\} = \{(x, y, z) \in M \mid z = 0, \ x = \sqrt{y^2 + z^2}\}$.
\item[$(3)$]
For any $r > 0$ the sphere $S(r)$ is a two-dimensional noncompact real-analytic manifold.
\end{itemize}
\end{theorem}
\begin{proof}
(1) follows from item (1) of Th. \ref{th:d}.

(2) follows from item (2) of Th. \ref{th:d}.

(3) follows since the mapping $\map{\Exp}{N}{\intt D}$ is real-analytic, proper, and diffeomorphic.
\end{proof}

\subsection{Convergence of the problem as $\eps \to 0$}
In this subsection $\eps$ is a variable parameter. Let us denote problem \eq{pr1}--\eq{pr3} as~$P_{\eps}$, and use the subscript $\eps$ to denote the objects related to this problem: the attainable set $\A_{\eps}$, the exponential mapping $\Exp_{\eps}$, the spheres $S_{\eps}(r)$. 

When $\eps \to 0$, the statement of the problem $P_{\eps}$ tends to the following statement:
\begin{align}
&\dot q = u_1 X_1 + u_2X_2, \qquad q \in M, \quad u_1 \geq |u_2|, \label{pr01}\\
&q(0) = q_0, \qquad q(t_1) = q_1, \\
&l = \int_0^{t_1} \sqrt{u_1^2 - u_2^2} dt
\to \max,
\end{align}
let us denote this problem as $P_0$.

$P_0$ is a sub-Lorentzian problem on the Heisenberg group studied in the papers \cite{groch4, groch6, sl_heis}. The aim of this subsection --- to study, in which sense the objects $\A_{\eps}$, $\Exp_{\eps}$, $S_{\eps}(r)$ for the problem $P_{\eps}$ tend respectively to the objects $\A_0$, $\Exp_0$, $S_0(r)$ of the problem $P_0$ as $\eps \to 0$.

\subsubsection{Convergence of control system and future cone}
The right-hand side of the control system for the problem $P_{\eps}$ tends as $\eps \to 0$ to the right-hand side of the problem $P_0$:
$$
\forall q \in M \ \forall u \in U \  \lim_{\eps\to 0} (u_1 X_1(q) + u_2 X_2(q) + \eps u_3X_3(q)) = u_1 X_1(q) + u_2 X_2(q).
$$

Consider the future cones of the problems $P_{\eps}$:
\begin{align*}
&C_{\eps} = \left\{v \in T_{q_0}M \mid dx^2 \geq dy^2 + \frac{dz^2}{\eps^2}, \ dx \geq 0 \right\}, \qquad \eps > 0,\\
&C_{0} = \left\{v \in T_{q_0}M \mid dx^2 \geq dy^2, \ dx \geq 0, \ dz = 0 \right\}.
\end{align*}

Denote by $\chi_A$ the characteristic function of a set $A$.

\begin{proposition}
\begin{itemize}
\item[$(1)$]
For any $0 \leq \eps_1 < \eps_2$ we have $C_{\eps_1} \subset C_{\eps_2}$.
\item[$(2)$]
For any $v \in T_{q_0}M$ we have $\lim_{\eps\to 0}\chi_{C_{\eps}}(v) = \chi_{C_0}(v)$.
\end{itemize}
\end{proposition}
\begin{proof}
(1) follows immediately from definitions of the cones $C_{\eps}$.

(2) If $v \in C_0$, then $\chi_{C_0}(v) = \chi_{C_{\eps}}(v) = 1$ for any $\eps > 0$.

If for some $\eps_0 > 0$ we have $v \in C_{\eps} \setminus C_0$, then $\chi_{C_0}(v) = 0$, $\chi_{C_{\eps}}(v) = 1$. Then $dz(v) \neq 0$ or $(dx^2-dy^2)(v) < 0$. If $dz(v) \neq 0$, then $\lim_{\eps\to 0} \frac{dz(v)}{\eps} = + \infty$, and for some $\eps_1 > 0$ we have $v \notin C_{\eps_1}$. Thus for all $\eps \in (0, \eps_1)$ we have $v \notin C_{\eps}$, i.e., $\chi_{C_{\eps}}(v) = 0$.

If $dz(v) = 0$, then $dy^2(v) > dx^2(v) \geq dy^2(v)$, a contradiction.

Finally, if $v \notin C_{\eps}$ for some $\eps > 0$, then $v \notin C_0$, and $\chi_{C_{\eps}}(v) = \chi_{C_0}(v) = 0$.
\end{proof}

\subsubsection{Convergence of the attainable set}

\begin{lemma}\label{lem:Ae1e2}
For any $0 \leq \eps_1 < \eps_2$ there holds the inclusion $\A_{\eps_1} \subset \A_{\eps_2}$.
\end{lemma}
\begin{proof}
As shown in work \cite{groch6}, 
$$
\A_0 = \{(x, y, z) \in M \mid x \geq 0, \ 4 |z| \leq x^2-y^2\}.
$$
Denote for any $\eps > 0$
\begin{align*}
&f_{\eps}(q) = |z| - \f_{\eps}(x, y), \\
&\f_{\eps}(x, y) = \frac{\eps^2}{2}(\sinh \tau + \tau), \qquad \tau = \arcosh\left(\frac{x^2-y^2}{2\eps^2}+1\right)
\end{align*}
and
\begin{align*}
&f_0(q) = |z| - \f_0(x, y), \\
&\f_0(x, y) = \frac{x^2-y^2}{4},
\end{align*}
then
$$
\A_{\eps} = \{q = (x, y, z)\in M \mid x \geq |y|, \ f_{\eps}(q) \leq 0 \}, \qquad \eps \geq 0.
$$
It is easy to see that $\f_{\eps}(x, y) > \f_0(x, y)$ for $x > |y|$, $\eps > 0$, thus $f_{\eps}(q) \leq f_0(q)$, $\eps > 0$, for $x \geq |y|$. Thus $\A_0 \subset \A_{\eps}$, $\eps > 0$. 

Now let $0 < \eps_1 < \eps_2$. It is easy to see that $\f_{\eps_1}(x, y) < \f_{\eps_2}(x, y)$, thus $f_{\eps_1}(q) > f_{\eps_2}(q)$ for $x > |y|$. Consequently, $\A_{\eps_1}\subset \A_{\eps_2}$.
\end{proof}

Boundaries of the attainable sets $\A_{\eps_1}\subset \A_{\eps_2}$ are shown for $\eps_1 = 0$, $\eps_2 = 1$ in Fig. \ref{fig:A01}, and for $\eps_1 = 1$, $\eps_2 = 2$ in Fig. \ref{fig:A12}.

\twofiglabelsize
{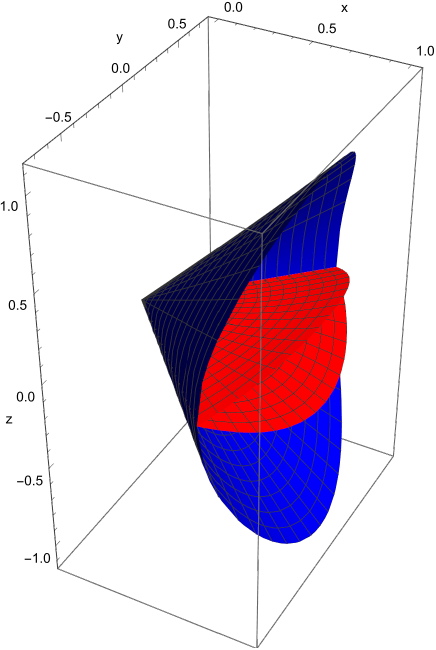}{Boundaries of the attainable sets $\A_{0}$ and $\A_{1}$}{fig:A01}{0.6}
{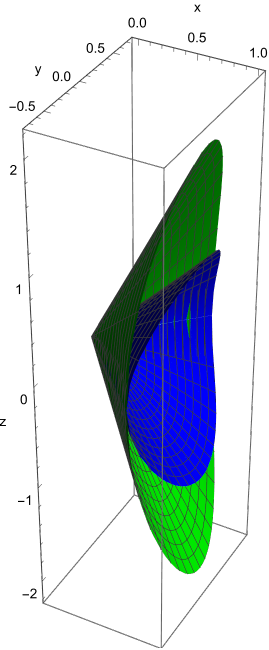}{Boundaries of the attainable sets $\A_{1}$ and $\A_{2}$}{fig:A12}{0.6}

\begin{theorem}
For any point $q \in M$ there holds the identity
\be{lim}
\lim_{\eps \to 0} \chi_{A_{\eps}}(q) = \chi_{A_{0}}(q).
\ee
\end{theorem}
\begin{proof}
If $q \in \A_0$, then $\chi_{\A_0}(q) = \chi_{\A_{\eps}}(q) = 1$ since $\A_0 \subset \A_{\eps}$, and equality \eq{lim} holds.

Let $q \notin \A_0$, then $\chi_{A_0}(q) = 0$. Fix any $\eps_0 > 0$. If $q \notin \A_{\eps_0}$, then $\chi_{\eps_0}(q) = 0$. For any $\eps \in (0, \eps_0)$ we have $\A_{\eps} \subset \A_{\eps_0}$, whence $\chi_{\A_{\eps}}(q) = 0$, and equality \eq{lim} holds.

Let $q = (x, y, z) \in \A_{\eps_0}$. Then $x \geq |y|$, $f_{\eps_0}(q) \leq 0$, $f_0(q) > 0$. Thus $\f_0(x, y) < |z| \leq \f_{\eps}(x, y)$. By virtue of the equality $\lim_{\eps \to 0} \f_{\eps}(x, y) = \f_{0}(x, y)$, there exists such $\eps_1 \in (0, \eps_0)$ that $\f_{\eps_1}(x, y) \leq |z|$. Then $f_{\eps}(q) > 0$, i.e., $q \notin \A_{\eps_1}$, and $\chi_{\A_{\eps_1}}(q) = 0$. For any $\eps \in (0, \eps_1)$ we have $\chi_{A_{\eps}}(q) = 0$, and equality \eq{lim} holds.
\end{proof}

So for $\eps \to 0$ there holds the relation $\A_{\eps} \to \A_{0}$ in the sense $\chi_{\A_{\eps}}(q) \to \chi_{\A_{0}}(q)$ for any $q \in M$. 

\subsubsection{Convergence of the exponential mapping}
As shown in work \cite{sl_heis}, the exponential mapping $\Exp_0$ in the problem $P_0$
is given as follows:
\begin{align*}
&\map{\Exp_0}{C_0 \times \R_{+t}}{M}, \qquad C_0 = T_{q_0}^*M \cap \{H_0 = -1/2\}, \\
&H_0(h_1, h_2,h_3) = \frac 12 (-h_1^2 + h_2^2), \\
&\map{\Exp_0}{(\theta, c, t)}{(x, y, z)}, \\
&h_1 = - \cosh \psi, \qquad h_2 = \sinh \psi, \qquad h_3 = c,
\intertext{where for $c = 0$}
&x = t \cosh \psi, \qquad y = t \sinh \psi, \qquad z = 0, \\
\intertext{and for $c \neq 0$}
&x = (\sinh(\psi + c t)-\sinh \psi)/c, \\
&y = (\cosh(\psi + ct)-\cosh \psi)/c, \\
&z = (\sinh ct = ct)/(2 c^2).
\end{align*}
In the problem $P_{\eps}$ normal extremal trajectories are parametrised by points of the hyperboloid
\begin{align*}
&C_{\eps} = T^*_{q_0}M \cap \{H_{\eps} = - 1/2\}, \\
&H_{\eps} = \frac 12 (-h_1^2+ h_2^2 + \eps^2 h_3^2), \\
&h_1 = - \cosh \theta, \qquad h_2 = \sinh \theta \cos \f, \qquad \eps h_3 = \sinh \theta \sin \f.
\end{align*}
Let us transfer the coordinates $(\psi, c)$ from the surface $C_0$ to the surface $C_{\eps}$ as follows:
$$
\sinh \theta \cos \f = \sinh \psi, \qquad \frac{\sinh \theta \sin \f}{\eps} = c,
$$
as a result we get functions $\theta(\psi, c)$ and $\f(\psi, c)$.

\begin{theorem}\label{th:Expe}
For any $(\psi, c, t) \in C_0 \times \R_{+t}$ we have
$$
\lim_{\eps\to 0} \Exp_{\eps}(\theta(\psi, c), \f(\psi, c), t) = \Exp_0(\psi, c, t).
$$
\end{theorem}
\begin{proof}
Immediate computation on the basis of explicit parametrisations of the exponential mappings.
\end{proof}

So for $\eps \to 0$ parametrisations of extremal trajectories of the problems $P_{\eps}$ tend to parametrisation of extremal trajectories for the problem $P_0$ in the sense of Theorem \ref{th:Expe}.

\subsubsection{Convergence of spheres}
In problems $P_{\eps}$ and $P_0$ spheres of radius $r > 0$ are parametrised by the exponential mappings:
\begin{align*}
&S_{\eps}(r) = \{\Exp_{\eps}(\theta, \f, r) \mid \theta \in \R, \ \f \in \R/(2 \pi \Z)\}, \qquad \eps > 0, \\
&S_{\eps}(0) = \{\Exp_{\eps}(\psi, c, r) \mid \psi, c \in \R\}.
\end{align*}

Let us recall several definitions on convergence of sequences of sets \cite{aubin}.
Let a family of sets $A_{\eps}$, $\eps \geq 0$, be given in a metric space $(M, \rho)$. The lower and upper limits of the family $A_{\eps}$ as $\eps \to 0$ are defined, respectively, as
$$
\liminf_{\eps \to 0} A_{\eps} := \{ q \in M \mid \lim_{\eps \to 0} \rho(q, A_{\eps}) = 0\}
$$
and
$$
\limsup_{\eps \to 0} A_{\eps} := \{ q \in M \mid \liminf_{\eps \to 0} \rho(q, A_{\eps}) = 0\},
$$
where $\rho(q, A) := \inf_{p \in A} \rho(q, p)$. A family $A_{\eps}$ is called lower (upper) semicontinuous as $\eps \to 0$ in the sense of Kuratowski if $A_0 \subset \liminf_{\eps\to 0} A_{\eps}$ (respectively, $A_0 \supset \limsup_{\eps\to 0} A_{\eps}$). Finally, a family of sets is continuous if it is lower and upper semicontinuous.

\begin{theorem}\label{th:Se}
Let $r > 0$. 
\begin{itemize}
\item[$(1)$]
If $\eps \to 0$, then parametrisation of the spheres $S_{\eps}(r)$ tends to parametrisation of the sphere in the sense that
$$
\lim_{\eps\to 0} \Exp_{\eps}(\theta, \f, r) = \Exp_0(\psi, c, r).
$$ 
\item[$(2)$]
The family $S_{\eps}$ is lower semicontinuous as $\eps \to 0$ in the sense of Kuratowski.
\end{itemize}
\end{theorem}
\begin{proof}
(1) follows from Th. \ref{th:Expe}.

(2) Let $q \in S_0(r)$, then $q = \Exp(\psi, c, r)$ for some $\psi, c \in \R$. Let $q_{\eps} = \Exp_{\eps}(\theta(\psi, c), \f(\psi, c), r) \in S_{\eps}(r)$. By Theorem \ref{th:Expe} we have $\lim_{\eps\to 0} q_{\eps} = q$, thus $S_0 \subset \liminf_{\eps\to 0} S_{\eps}$.
\end{proof}

So we proved that the family of spheres $S_{\eps}(r)$ tends to the sphere $S_0(r)$ as $\eps \to 0$ in the sense of Th. \ref{th:Se}.

Figures \ref{fig:S0}, \ref{fig:S01} suggest that $S_{\eps}$ is continuous as $\eps \to 0$ in the sense of Kuratowski.

\twofiglabelsize
{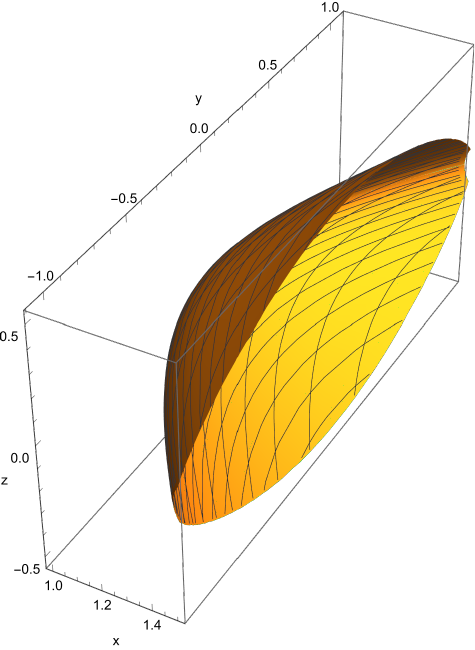}{The sphere $S_0(1)$}{fig:S0}{0.6}
{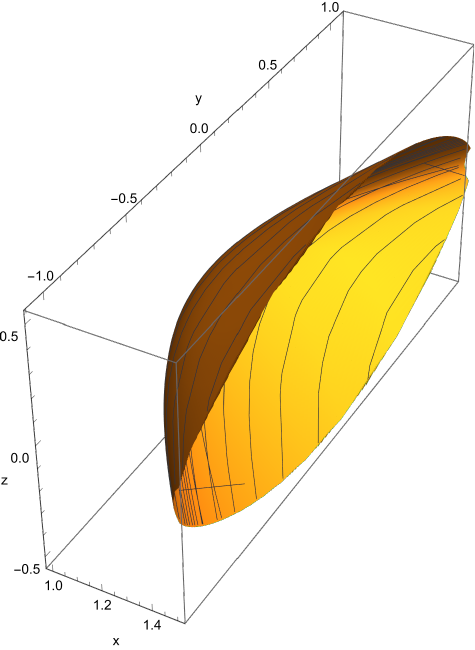}{The sphere $S_{0.1}(1)$}{fig:S01}{0.6}

\section{The second family of Lorentzian problems} \label{sec:second}
\subsection{Problem statement}
Consider the following family of Lorentzian problems on the Heisenberg group:
\begin{align}
&\dot q = u_1 X_1 + u_2 X_2 + \eps u_3 X_3, \qquad q \in M, \quad \eps > 0, \label{pr31}\\
&u \in U = \{u_3 \geq \sqrt{u_1^2+u_2^2}\}, \label{pr22}\\
&q(0) = q_0, \qquad q(t_1) = q_0, \label{pr23}\\
&l = \int_{0}^{t_1} \sqrt{u_3^2 - u_1^2 - u_3^2} dt \to \max. \label{pr25}
\end{align}
A characteristic property of these problems is that the union $C \cup C^-$ of the future cone
$$
C = \{ v \in T_{q_0}M \mid g(v) \leq  0\}, \qquad g = \om_1^2 + \om_2^2 - \frac{\om_3^2}{\eps^2}
$$
and the past cone
$
C^- = - C
$
contains in its interior the commutant $\R X_3(q_0)$ of the Heisenberg algebra, unlike the first family considered in Section \ref{sec:first}.
 We will show in this section that this condition changes radically the properties of a problem.

\subsection{Pontryagin maximum principle in the geometric form}
Application of the Pontryagin maximum principle in the geometric form \cite{notes} gives the following statement: if a trajectory $q(t)$, $t \in [0, t_1]$, comes to the boundary of the attainable set, then the point $q(t_1)$ belongs to the surface $S = \{(x, y, z) \in M \mid h_2 + h_2^2 - \eps^2 h_3^2 = - 1, \ -h_3 \leq 0 \}$, where 
\begin{align*}
&x = (h_1 \sin h_3 + h_2 (\cos h_3-1))/h_3, \\
&y = (h_2 \sin h_3 - h_1 (\cos h_3 - 1))/h_3, \\
&z = \frac{\eps^2 h_3^2-1}{2 h_3^2}(h_3 - \sin h_3) - \eps^2 h_3^2.
\end{align*}
The surface $S$ for $\eps = 1$ is shown in Fig. \ref{fig:S2}. It does not divide $M$ into two connected domains. It is a marker of global controllability of the system, see Subsec. \ref{subsec:attain2}.

\onefiglabelsizen{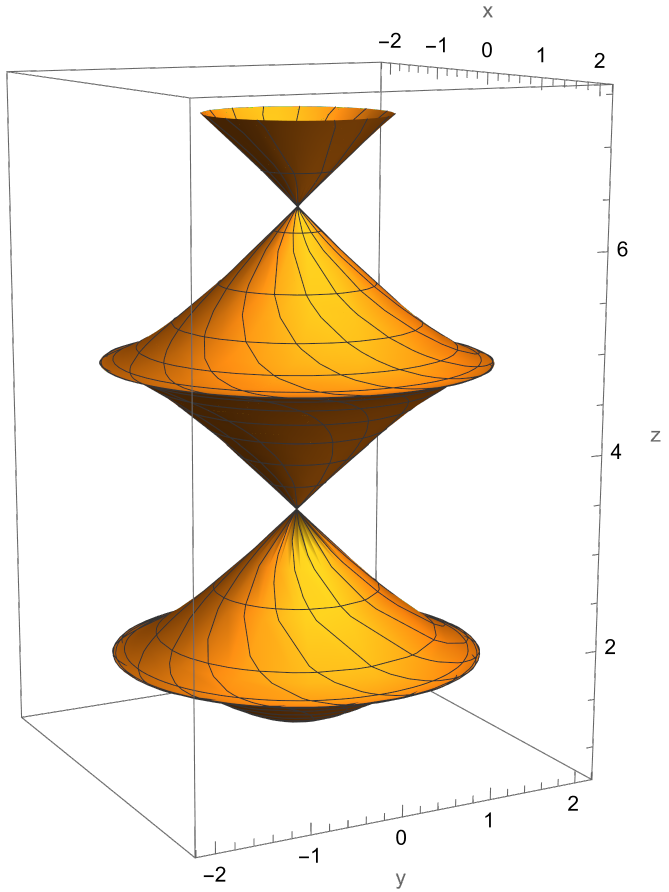}{The surface $S$  for $\eps = 1$}{fig:S2}{6}

\subsection{Global controllability}\label{subsec:attain2}
\begin{theorem}\label{th:attain2}
System \eq{pr31}--\eq{pr23} has a periodic trajectory and is globally controllable.
\end{theorem}
\begin{proof}
Consider first a constant admissible control: $u = (u_1, u_2, u_3) \in U$. Then system \eq{pr31}--\eq{pr23} has the trajectory
\be{tr123}
x = u_1 t, \qquad y = u_2 t, \qquad z = \eps u_3 t.
\ee
Such trajectories fill the set $\{z \geq \eps \sqrt{x^2+y^2}\}$, thus
$$
\A \supset \{(x, y, z) \in M \mid z \geq \eps \sqrt{x^2+y^2}\}.
$$
Considering trajectories \eq{tr123} in reverse time, we get the inclusion
\be{incl-}
\A^- \supset \{(x, y, z) \in M \mid z \leq - \eps \sqrt{x^2+y^2}\}.
\ee

Now take arbitrary $0 < t_1 < t_2$ and consider the admissible control
\be{u(t)}
u(t) = 
\begin{cases}
(1, 0, 1), \qquad & t \in [0, t_1], \\
(0, -1, 1), \qquad & t \in (t_1, t_2].
\end{cases}
\ee
Immediate integration of system \eq{pr31}--\eq{pr23} gives the following:
\begin{align*}
&q(t_1) = (t_1, 0, \eps t_1), \\
&x(t_2) = t_1, \qquad y(t_2) = t_1 - t_2, \qquad z(t_2) = \eps t_2 + \left(\eps - \frac{t_1}{2}\right)(t_2-t_1).
\end{align*}
If $t_1 > 2 \eps$, then we have $\frac{dz}{dy} > \eps$ for $t \in (t_1, t_2]$, thus for sufficiently big $t_2$ we get the inclusion  
$q(t_2) \in \{z \leq - \eps \sqrt{x^2+y^2}\}$, whence with account of \eq{incl-} we get $q(t_2) \in \A^-$. Moreover, we can assume that $q(t_2) \in \intt \A^-$. Consequently, there exists $t_3 > t_2$ such that the control $u(\cdot)$ with conditions \eq{u(t)} and $u(t) =  u_3 \equiv \const \in U$ for $t \in (t_2, t_3]$ gives a periodic trajectory $q(t)$: $q(0) = q(t_3) = q_0$.

The inclusion $q(t_2) \in \intt \A^-$ implies that $q_0 \in \intt \A$, i.e., system \eq{pr31}--\eq{pr23} is locally controllable at the point $q_0$. Then Theorem 2.7 \cite{obzor} implies that this system is globally controllable as well.
\end{proof}

A periodic trajectory for system \eq{pr31}--\eq{pr23} is shown in Fig. \ref{fig:period}.

\onefiglabelsizen{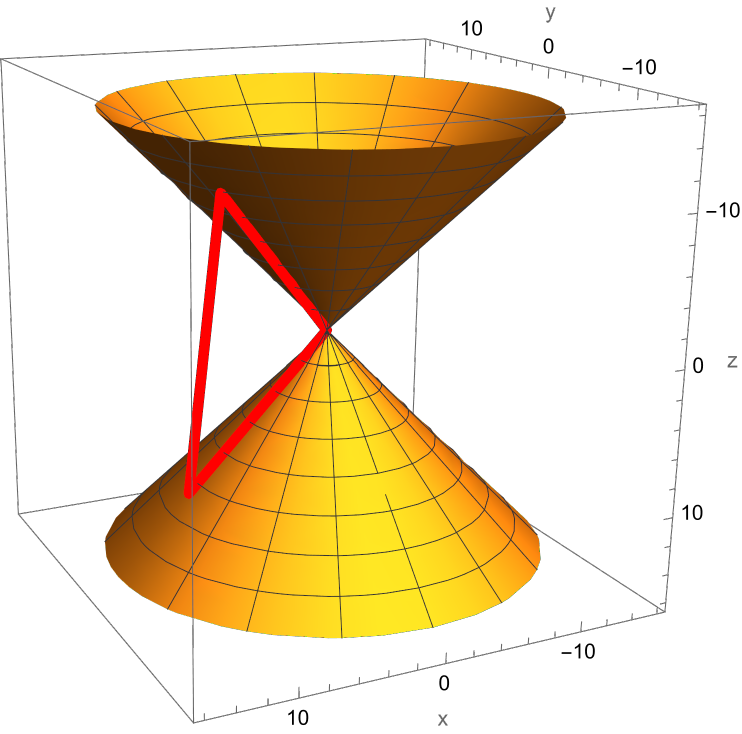}{Periodic trajectory for system \eq{pr31}--\eq{pr23}}{fig:period}{6}

\subsection{Non-existence of optimal trajectories}
Theorem \ref{th:attain2} implies the following
\begin{corollary}
There are no optimal trajectories in problem \eq{pr31}--\eq{pr25}.
\end{corollary}
\begin{proof}
It follows from the proof of Theorem \ref{th:attain2} that system \eq{pr31}--\eq{pr23} has a periodic trajectory $q(\cdot)$ with a nontrivial timelike arc, thus $l(q(\cdot))>0$.   Thus for any point $q_1 \in M$ there exists an admissible trajectory of arbitrarily big length~$l$. So there are no optimal trajectories in problem \eq{pr31}--\eq{pr25}. 
\end{proof}

\subsection{Extremal trajectories}
\subsubsection{Normal trajectories}
It follows from the Pontryagin maximum principle \cite{PBGM, notes} that arclength para\-met\-ri\-sed normal extremals are trajectories of the Hamiltonian vector field $\vec H$ with the Hamiltonian $H = \frac 12 (h_1^2 + h_2^2 - \eps^2 h_3^2)$ lying at the level surface $\{H = - \frac 12\}$. The vertical subsystem of this system (for adjoint variables) has the form
\begin{align*}
&\dot h_1 = - h_2 h_3, \qquad \dot h_2 = h_1 h_3, \qquad \dot h_3 = 0, \\
&h_1^2+h_2^2 - \eps^2 h_3^2 = - 1, \qquad h_3 < 0.
\end{align*}
Thus
\begin{align*}
&h_1 = h_1^0 \cos \tau - h_2^0 \sin \tau, \qquad \tau = h_3t, \\
&h_2=h_2^0 \cos \tau + h_1^0 \sin \tau.
\end{align*}
The horizontal subsystem of the Hamiltonian system (for state variables) is
\begin{align*}
&\dot x = h_1, \\
&\dot y = h_2, \\
&\dot z = -h_1 \frac y2 + h_2 \frac x2 - \eps^2 h_3,
\end{align*}
whence
\begin{align}
&x = (h_1^0 \sin \tau + h_2^0 (\cos \tau - 1))/h_3, \label {xt2} \\
&y = (h_2^0 \sin \tau - h_1^0 (\cos \tau - 1))/h_3, \label{yt2} \\
&z = \frac{\eps^2h_3^2-1}{2h_3^2} (\tau - \sin \tau) - \eps^2 \tau. \label{zt2}
\end{align}

\subsubsection{Abnormal trajectories}
It follows from the Pontryagin maximum principle that optimal trajectories have the form \eq{xt2}, \eq{yt2}, 
$$
z = - \frac{\eps^2}{2}(\tau + \sin \tau).
$$

\subsection{Conjugate points}
Let us parametrise the surface $E = \{h_1^2+h_2^2-\eps^2h_3^2 = -1, \ h_3 < 0\}$ as follows:
\begin{align*}
&h_1 = \sinh \theta \cos \f, \qquad h_2 = \sinh \theta \sin \f, \qquad \eps h_3 = - \cosh \theta, \\
&\theta \in \R, \quad \f \in \R/(2 \pi \Z). 
\end{align*}
Normal extremal trajectories \eq{xt2}--\eq{zt2} define the exponential mapping 
$$
\map{\Exp}{(T_{q_0}^*M \cap E) \times \R_{+}}{M}, \qquad (\theta, \f, \tau) \mapsto (x, y, z).
$$
Conjugate points are zeros of the Jacobian $J = \frac{\partial (x, y, z)}{\partial (\theta, \f, \tau)}$. Immediate computation on the basis of formulas \eq{xt2}--\eq{zt2} gives
\begin{align*}
&J = - \eps^4 \frac{\sinh \theta}{\cosh^3 \theta} f(\theta, \tau), \\
&f(\theta, \tau) = \frac{2(1 - \cos \tau)}{\cosh^2 \theta} + \tau \sin \tau \tanh^2 \theta.
\end{align*}
Thus we have in the preimage of the exponential mapping
\be{conj_preim}
J = 0 \iff \theta = 0, \ \tau = 2 \pi n, \ n \in \N.
\ee
The corresponding points in the image of the exponential mapping are
\be{conj_im}
x = y = 0, \qquad z = 2 \pi n \eps^2, \quad n \in N 
\ee

We proved the following

\begin{proposition}\label{prop:conj2}
Conjugate points in problem \eq{pr31}--\eq{pr24} have the form \eq{conj_im}. The first conjugate time for an extremal trajectory $\Exp(\theta, \f, t)$ is $t = \frac{2 \pi}{|h_3|}$.
\end{proposition}

\subsection{Application to general left-invariant problems on the Heisenberg group}

Let $\tilde g$ be a left-invariant Lorentzian structure on the Heisenberg group with the future cone
$$
\tilde C = \{v \in T_{q_0}M \mid \tilde g(v) \leq 0\}
$$
and the past cone
$
\tilde C^- = - \tilde C
$.
We say that a subspace $L \subset T_{q_0}M$ is contained inside of the future and past cones if
$$
L \setminus \{0\} \subset \intt (\tilde C \cup \tilde C^-).
$$

Recall that the commutant of the Heisenberg algebra is $\R X_3(q_0)$.

\begin{theorem}
\begin{itemize}
\item[$(1)$]
If $\tilde C \cap (\R X_3(q_0))  = \{0\}$, then the Lorentzian structure $\tilde g$ has length maximizers between any points $q_0$ and $q_1$ such that $q_0$ can be connected by nonspacelike trajectory with $q_1$. 
\item[$(2)$]
If $\R X_3(q_0)$ is contained inside of the future and past cones of $\tilde g$ then the Lorentzian structure $\tilde g$ has periodic nonspacelike trajectories, is globally controllable and  has no length maximizers.
\end{itemize}
\end{theorem}
\begin{proof}
(1) follows immediately from Corollary~1 \cite{sl_exist_pod}.

(2) For sufficiently small $\eps > 0$ the future cone of the system \eq{pr31}--\eq{pr23} is contained in the future cone $\tilde C$ of the structure $\tilde g$, thus any trajectory of system \eq{pr31}--\eq{pr23} is a nonspacelike curve of structure $\tilde g$. The system \eq{pr31}--\eq{pr23} has a periodic admissible trajectory, is globally controllable, and has no length maximizers. Thus the Lorentzian structure  $\tilde g$ inherits these properties.
\end{proof}


\begin{thebibliography}{9}

 \bibitem{groch4}
M. Grochowski,
On the Heisenberg sub-Lorentzian metric on $\R^3$,
GEOMETRIC SINGULARITY THEORY,
BANACH CENTER PUBLICATIONS,   
INSTITUTE OF MATHEMATICS,
POLISH ACADEMY OF SCIENCES,
WARSZAWA, 
vol. 65,
2004.

 \bibitem{groch6}
M. Grochowski,
Reachable sets for the Heisenberg sub-Lorentzian structure on $\R^3$. An estimate for the distance function.
{\em Journal of Dynamical and Control Systems},
vol. 12,
2006,
2, 145--160. 

 \bibitem{sl_heis}	Yu. L. Sachkov, E.F. Sachkova, Sub-Lorentzian distance and spheres on the Heisenberg group, Journal of Dynamical and Control Systems volume 29, pages 1129--1159 (2023)


\bibitem{aubin}
J.-P. Aubin, H. Frankowska, {\em Set-Valued Analysis}, Springer, 2009.

\bibitem{sl_exist_pod}
  A. V. Podobryaev, Existence of Length Maximizers for Left-Invariant 3D Contact Sub-Lorentzian Structures, {\em submitted}.

\bibitem{notes}
A.A.~Agrachev, Yu. L. Sachkov,  Control Theory from the Geometric Viewpoint,
Springer-Verlag, 2004

\bibitem{PBGM}
L.S.~Pontryagin, V.G.~Boltyanskii, R.V.~Gamkrelidze, E.F.~Mishchenko, 
{\em The mathematical theory of optimal processes}, Wiley Interscience, 1962. 

\bibitem{had}
Krantz~S.~G., Parks~H.~R., \textit{The Implicit Function Theorem: History, Theory, and Applications}, Birkauser, 2001.

\bibitem{sl_exist}
Yu. L. Sachkov, Existence of sub-Lorentzian length maximizers, Differential Equations. 59, 12, 1702--1709 (2023)

\bibitem{obzor}
Yu. L. Sachkov,
Control Theory 
on Lie Groups,
{\em Journal of Mathematical Sciences},   
  Vol. 156, No. 3, 2009, 381-439.

\end{thebibliography}
\end{document}